\documentclass[12pt]{elsarticle}
\usepackage{amsthm,amsmath,amstext,amssymb,
mathrsfs,extsizes,MnSymbol}

\usepackage[all]{xy}

\usepackage[active]{srcltx}
 \newtheorem{theorem}{Theorem}
 
 \newtheorem{corollary}{Corollary}

\theoremstyle{definition}
\newtheorem{definition}{Definition}

\theoremstyle{remark}
\newtheorem{remark}{Remark}

\renewcommand{\le}{\leqslant}
\renewcommand{\ge}{\geqslant}

\DeclareMathOperator{\sign}{sign}

\newcommand{\mat}[1]{\begin{bmatrix}
 #1  \end{bmatrix}}

\newcommand{\mc}{\mathcal}
\newcommand{\ff}{\mathbb F}
\newcommand{\pp}{\mathbb P}
\newcommand{\cc}{\mathbb C}
\newcommand{\rr}{\mathbb R}

\newcommand{\un}{\underline}

\newcommand{\st}{\upY}

\begin{document}

\title{Congruence of matrix spaces, matrix tuples, and multilinear maps\footnote{G.R. Belitskii, V. Futorny, M. Muzychuk, V.V. Sergeichuk,l Congruence of matrix spaces, matrix tuples, and multilinear maps, Linear Algebra Appl. 609 (2021) 317-331}}

\author[bel]{Genrich R. Belitskii}
\address[bel]{Department of Mathematics,
Ben-Gurion University of the Negev,
Beer-Sheva, Israel}
\ead{genrich@cs.bgu.ac.il}

\author[fut]{Vyacheslav Futorny}
\ead{futorny@ime.usp.br}
\address[fut]{Department of Mathematics, University of S\~ao Paulo, Brazil}

\author[bel]{Mikhail Muzychuk}
\ead{misha.muzychuk@gmail.com}

\author[ser]{Vladimir~V.~Sergeichuk\corref{cor}}
\ead{sergeich@imath.kiev.ua}
\address[ser]{Institute of Mathematics, Tereshchenkivska 3,
Kiev, Ukraine}

\cortext[cor]{Corresponding author.}

\begin{abstract}
Two matrix vector spaces $\mc V,\mc W\subset \cc^{n\times n}$ are said to be equivalent if
$S\mc VR=\mc W$ for some nonsingular $S$ and $R$.
These spaces are congruent if $R=S^T$. We prove that if all matrices in $\mc V$ and $\mc W$ are symmetric, or all matrices  in $\mc V$ and $\mc W$  are skew-symmetric, then $\mc V$ and $\mc W$ are congruent if and only if they are equivalent.

Let $\mc F: U\times\dots\times U\to V$ and $\mc G: U'\times\dots\times U'\to V'$ be symmetric or skew-symmetric $k$-linear maps over $\cc$. If there exists a set of linear bijections $\varphi_1,\dots,\varphi_k:U\to U'$ and
$\psi:V\to V'$ that transforms $\mc F$ to $\mc G$, then there exists such a set with $\varphi_1=\dots=\varphi_k$.
\end{abstract}

\maketitle

\section{Introduction and main results}

If two pairs of complex $n\times n$ matrices $(M,M')$ and $(N,N')$ are congruent, then they are equivalent; that is,
\begin{align*}
\exists &\text{ nonsingular } S:\ S(M,M')S^T:=(SMS^T,SM'S^T)=(N,N')\\
\Longrightarrow\quad
\exists &\text{ nonsingular }  P\text{ and } Q:\ P(M,M')Q=(N,N').
\end{align*}
Mal'cev
writes in \cite[Chapter VI]{mal} that quite unexpectedly the converse is also true under certain conditions:
\begin{equation}
\label{bgm}
\parbox[c]{0.84\textwidth}
{\emph{Let $M$ and $N$ be either both symmetric or both skew-symmetric, and let $M'$ and $N'$ be either both symmetric or both skew-symmetric. If $(M,M')$ and $(N,N')$ are equivalent, then they are congruent.}}
\end{equation}

Since each square complex matrix $A$ can be expressed as the sum of the symmetric matrix $(A+A^T)/2$ and the skew-symmetric matrix$(A-A^T)/2$, the statement \eqref{bgm} implies that
\begin{equation}
\label{bgvm}
\parbox[c]{0.84\textwidth}
{two complex matrices $A$ and $B$ are congruent if and only if $(A,A^T)$ and $(B,B^T)$ are equivalent.}
\end{equation}
This fact admits to derive
a canonical form of complex matrices under congruence from the Kronecker canonical form of matrix pencils; see \cite[Section 2]{h-s1}.
The statement \eqref{bgvm} is extended in \cite{roi,ser} (see also \cite{h-s1,h-s2,ser1})  to arbitrary systems of forms and linear maps.

If $(A,A^T)$ and $(B,B^T)$ are equivalent, then $P(A,A^T)Q=(B,B^T)$ for some nonsingular $P$ and $Q$, and so $(PAQ,Q^TAP^T)=(B,B)$. Taking $R:=Q^T$, we rewrite \eqref{bgvm} as follows: \begin{equation}
\label{bkm}
\parbox[c]{0.84\textwidth}
{two complex matrices $A$ and $B$ are congruent if and only if $PAR^T=RAP^T=B$ for some nonsingular $P$ and $R$.}
\end{equation}

This paper is a continuation of the article \cite{bel-ser_mul}, in which the statement \eqref{bkm} is extended to multilinear forms.
We extend \eqref{bkm} to matrix spaces, matrix tuples, and multilinear maps.
The main results are Theorem \ref{t1h} about congruence of matrix tuples and Theorem \ref{ther} about congruence of $k$-linear maps. In Remark \ref{bvc} we show that Theorem \ref{ther} in the case of bilinear maps (i.e., for $k=2$) follows from Theorem \ref{t1h}.

\subsection{Matrix tuples up to congruence and *congruence}
\label{s1.2a}

Let $\ff$ be a field or skew field with a fixed \emph{involution} $a\mapsto \bar a$; that is, a bijection $\ff\to\ff$   (which can be the identity if $\ff$ is a field)
such that $\overline{a+b}=\bar a+\bar b$,
$\overline{ab}=\bar b\bar a$, and $\bar{\bar
a}=a$ for all $a,b\in\ff$. If $A$ is a matrix over $\ff$, then $A^*:=\bar A^T$.

\begin{definition}\label{nhj}
Two $t$-tuples
\[\mc A=(A_1,\dots,A_t),\qquad \mc B=(B_1,\dots,B_t)\]
of $n\times n$ matrices over $\ff$ are
\begin{itemize}
  \item  \emph{symmetrically equivalent} if
$S\mc AR^T=R\mc AS^T=\mc B$,

  \item   \emph{*symmetrically equivalent} if
$S\mc AR^*=R\mc AS^*=\mc B$,

  \item \emph{congruent} if
$S\mc AS^T=\mc B$, and
\emph{*congruent} if
$S\mc AS^*=\mc B$
\end{itemize}
for some nonsingular $S$ and $R$.
\end{definition}

Recall that a \emph{real closed field} $\pp$ is a field whose algebraic closure is two-dimensional over $\pp$; for example, $\pp=\rr$. The characteristic of each real closed field is zero.

Let $p,q\in\{0,1,2,\dots\}$.
A \emph{$(p,q)$
block diagonal matrix} is a matrix of the form
\[
\mat{A&0\\0&B},\qquad
\begin{matrix}
A\text{ is }p\times p,\\
B\text{ is }q\times q.\\
\end{matrix}
\]

The following theorem generalizes \eqref{bkm}.

\begin{theorem}[proved in Section \ref{s2}]\label{t1h}
Let $\mc A$ and $\mc B$ be two
$t$-tuples
of $n\times n$ matrices over $\mathbb F$.

\begin{itemize}
  \item[\rm(a)]
Let $\ff$ be an algebraically closed field of characteristic different from $2$. Then ${\mc A}$ and ${\mc B}$ are symmetrically equivalent if and only if they are  congruent.

  \item[\rm(b)]
Let $\ff$ be a real closed field. Then ${\mc A}$ and ${\mc B}$ are symmetrically equivalent if and only if there exists a $t$-tuple $\mc C$ of $(p,n-p)$ block diagonal matrices such that $\mc A$ is congruent to $\mc C$ and $(I_p\oplus -I_{n-p})\mc C$
is congruent to ${\mc B}$.

\item[\rm(c)]
Let $\ff$ be an algebraically closed field or the skew field of real quaternions. Let a nonidentity involution on $\ff$ be fixed. Then ${\mc A}$ and ${\mc B}$ are *symmetrically equivalent if and only if
there exists a $t$-tuple $\mc C$ of $(p,n-p)$ block diagonal matrices such that $\mc A$ is *congruent to $\mc C$ and $(I_p\oplus -I_{n-p})\mc C$
is *congruent to ${\mc B}$.
\end{itemize}
\end{theorem}

\subsection{Matrix tuples up to weak congruence and *congruence}
\label{s1.2b}

\begin{definition}\label{nmj}
Let $\ff$ be a field.
Two $t$-tuples
$\mc A$ and $\mc B$
of $n\times n$ matrices over $\ff$ are
\emph{weakly symmetrically equivalent}, \emph{weakly *symmetrically equivalent}, \emph{weakly symmetrically congruent}, or \emph{weakly symmetrically *congruent} if there exists a nonsingular $t\times t$ matrix $\Lambda=[\lambda_{ij}]$ over $\mathbb F$ such that $\mc A$ is symmetrically equivalent,  *symmetrically equivalent, symmetrically congruent, or  symmetrically *congruent to the $t$-tuple
\begin{equation}\label{mjg}
\mc C:=(\lambda_{11}B_1
+\dots+\lambda_{1t}B_t,\ \dots,\ \lambda_{t1}B_1
+\dots+\lambda_{tt}B_t).
\end{equation}
\end{definition}

If
$\mc A$ and $\mc B$
are weakly *symmetrically equivalent, then $\mc A$ is *symmetrically equivalent to $\mc C$ of the form \eqref{mjg}. By Theorem \ref{t1h}, $\mc A$ is *congruent to $\mc C$, and so $\mc A$ and $\mc B$
are weakly *congruent, which ensures the following corollary.

\begin{corollary}[of Theorem \ref{t1h}]\label{t1n1}
Let $\mc A$ and $\mc B$ be two
$t$-tuples
of $n\times n$ matrices over $\mathbb F$.

\begin{itemize}
  \item[\rm(a)]
Let $\ff$ be an algebraically closed field of characteristic different from $2$. Then ${\mc A}$ and ${\mc B}$ are weakly  symmetrically equivalent if and only if they are weakly  congruent.

  \item[\rm(b)]
Let $\ff$ be a real closed field. Then ${\mc A}$ and ${\mc B}$ are weakly symmetrically equivalent if and only if there exists a $t$-tuple $\mc C$ of $(p,n-p)$ block diagonal matrices such that $\mc A$ is congruent to $\mc C$ and $(I_p\oplus -I_{n-p})\mc C$
is weakly  congruent to ${\mc B}$.

\item[\rm(c)]
Let $\ff$ be an algebraically closed field with a fixed nonidentity involution. Then ${\mc A}$ and ${\mc B}$ are weakly *symmetrically equivalent if and only if
there exists a $t$-tuple $\mc C$ of $(p,n-p)$ block diagonal matrices such that $\mc A$ is *congruent to $\mc C$ and $(I_p\oplus -I_{n-p})\mc C$
is weakly *congruent to ${\mc B}$.
\hfill\hbox{\qedsymbol}
\end{itemize}
\end{corollary}

\begin{remark}\label{co1d}
The word ``symmetrically'' in (a) and (b) of Theorem \ref{t1h} and Corollary \ref{t1n1} can be deleted if the matrices in $\mc A$ and $\mc B$ are all symmetric or all skew-symmetric. The word ``*symmetrically'' in (c) can be deleted if the matrices in $\mc A$ and $\mc B$ are all Hermitian.
\end{remark}

The problem of classifying matrix $t$-tuples up to weak congruence arises in many branches of mathematics:
\begin{itemize}
  \item In multilinear algebra: if $U$ and $V$ are vector spaces with bases $e_1,\dots,e_n$ and $f_1,\dots,f_t$, then each bilinear map $\mc F:U\times U\to V$ is given by a $t$-tuple $(A_1,\dots,A_t)$ of $n\times n$ matrices as follows:
\[
\mc F(x,y)=([x]_e^TA_1[y]_e)f_1
+\dots+ ([x]_e^TA_t[y]_e)f_t
\]
in which $[x]_e$ and $[y]_e$ are the coordinate vectors of $x,y\in U$. We can reduce $(A_1,\dots,A_t)$ by transformations of weak congruence changing the bases in $U$ and $V$.

  \item In the theory of tensors:  by \cite{fut1}, the problem of classifying matrix $t$-tuples up to weak equivalence contains the problem of classifying
an arbitrary system of tensors of order at most three.

  \item In the theory of groups: each finite $p$-group $G$ of exponent $p\ne 2$ with central commutator subgroup $G'$ is given by a skew-symmetric bilinear map
      \begin{align*}
G/G' \times G/G' &\to G'\\ (g_1G',g_2G')&\mapsto [g_1, g_2]:=g_1^{-1}g_2^{-1}g_1g_2
      \end{align*}
between the vector spaces $G/G'$ and $G'$  over the field $\ff_p$ with $p$ elements, and so $G$ is given by a tuple of skew-symmetric matrices over $\ff_p$ defined up to weak congruence (see \cite[Lemma 5.2]{bel-dm} and \cite{ser_metab}). The problem of classifying nilpotent Chernikov $p$-groups with elementary top also reduces to the problem of classifying tuples of skew-symmetric matrices over $\ff_p$ up to weak congruence; see \cite{dro}.

\item The problem of classifying local commutative algebras (respectively, Lie algebras with central commutator subalgebra) over a field of characteristic different from 2 contains the problem of classifying,   up to weak congruence, of matrix $t$-tuples in which all matrices are symmetric (respectively, skew-symmetric); see \cite[Section 4]{bel-dm} and \cite{bel-lip,fut}.
\end{itemize}

Some classes of matrix 2-tuples are classified up to weak equivalence in \cite{bel-ber1,bel-ber,ehr}.
By \cite{bel-dm,bel-lip}, the problem of classifying matrix 3-tuples up to weak equivalence is wild, and so it contains the problems of classifying each system of linear maps and  representations of each finite dimensional algebra; see \cite{bel-ser_compl} and \cite[Proposition 9.14]{bar}.

\subsection{Matrix spaces up to congruence and *congruence}
\label{s1.1}

\begin{definition}\label{nhz}
Let $\ff$ be a field. Two matrix vector spaces $\mc V,\mc W\subset\ff^{n\times n}$ over $\ff$ are
\emph{equivalent} if
$S\mc VR=\mc W$,
\emph{congruent} if
$S\mc VS^T=\mc W$, and
\emph{*congruent}  with respect to an involution $a\mapsto \bar a$ on $\ff$ if
$S\mc VS^*=\mc W$
for some nonsingular $S$ and $R$.
\end{definition}

\begin{corollary}[of Theorem \ref{t1h}]
\label{t1}
Let $\mc V$ and $\mc W$ be two vector spaces of $n\times n$ matrices over $\mathbb F$.

\begin{itemize}
  \item[\rm(a)]
Let $\ff$ be an algebraically closed field of characteristic different from $2$. Let the matrices of $\mc V$ and $\mc W$ be all symmetric or all skew-symmetric. Then ${\mc V}$ and ${\mc W}$ are equivalent if and only if they are congruent.

  \item[\rm(b)]
Let $\ff$ be a real closed field. Let the matrices of $\mc V$ and $\mc W$ be all symmetric or all skew-symmetric. Then ${\mc V}$ and ${\mc W}$ are equivalent if and only if there exists a space $\mc U$ of $(p,n-p)$ block diagonal matrices such that $\mc V$ is congruent to $\mc U$ and $(I_p\oplus -I_{n-p})\mc U$
is congruent to ${\mc W}$.

\item[\rm(c)]
Let $\ff$ be an algebraically closed field with a fixed nonidentity involution. Let all matrices of $\mc V$ and $\mc W$ be Hermitian. Then ${\mc V}$ and ${\mc W}$ are equivalent if and only if
there exists a space $\mc U$ of $(p,n-p)$ block diagonal Hermitian matrices such that $\mc V$ is *congruent to $\mc U$ and $(I_p\oplus -I_{n-p})\mc U$
is *congruent to ${\mc W}$.
\end{itemize}
\end{corollary}

\begin{proof}
Let us prove (a), the statements (b) and (c) are proved analogously.
Let $\mc V$ and $\mc W$ be equivalent; that is,
$S\mc VR=\mc W$ for some nonsingular matrices $S$ and $R$. Let $A_1,\dots,A_t$ be a basis of $\mc V$.
Then the matrices
\[
B_1:=SA_1R,\ \dots,\ B_t:=SA_tR
\]
form a basis of $\mc W$. The tuples $(A_1,\dots,A_t)$ and $(B_1,\dots,B_t)$ are equivalent, and so they are congruent by Remark \ref{co1d}. Hence, $\mc V$ and $\mc W$ are congruent.
\end{proof}

\subsection{Multilinear maps}
\label{s1.3}

\begin{definition}
Two $k$-linear maps
\begin{equation}\label{8ut}
\mc F: U\times\dots\times U\to V,\qquad\mc G: U'\times\dots\times U'\to V'
\end{equation}
with $k\ge 2$ over a field $\ff$ are
\begin{itemize}
  \item \emph{equivalent} if
there exist linear bijections $\varphi_1,\dots,\varphi_k:U\to U'$ and
$\psi:V\to V'$ such that
\begin{equation}\label{jnd7}
\forall u_1,\dots,u_k\in U:\ \mc G(\varphi_1u_1,
\dots,\varphi_ku_k )=\psi \mc F(u_1,\dots,u_k),
\end{equation}
which means that the diagram
\[
\xymatrix@C=30pt{
U\ar[d]_{\varphi_1}
\ar@{}[r]|{\hspace{-25pt}
{\displaystyle\times\dots\times}
\hspace{-25pt}}
&U\ar[d]_{\varphi_k}
\ar[rr]^{\mc F}&&V
\ar[d]_{\psi}
                           \\
U'\ar@{}[r]|{\hspace{-25pt}
{\displaystyle\times\dots\times}
\hspace{-25pt}}
&U'\ar[rr]^{\mc G}&&V'
}
\]
is commutative;

  \item
\emph{symmetrically equivalent}
if there exist linear bijections $\varphi_1,\dots,\varphi_k:U\to U'$ and
$\psi:V\to V'$ such that
\begin{equation}\label{jnd}
\forall \sigma \in S_k,\, u_1,\dots,u_k\in U:\ \mc G(\varphi_{\sigma(1)}u_1,
\dots,\varphi_{\sigma(k)}u_k )=\psi \mc F(u_1,\dots,u_k),
\end{equation}
in which $S_k$ is the permutation group;

  \item \emph{congruent} if
there exist linear bijections $\varphi:U\to U'$ and
$\psi:V\to V'$ such that
\begin{equation}\label{jnd8}
\forall u_1,\dots,u_k\in U:\ \mc G(\varphi u_1,
\dots,\varphi u_k )=\psi \mc F(u_1,\dots,u_k).
\end{equation}
\end{itemize}
\end{definition}
The \emph{direct sum} of  $k$-linear maps
\eqref{8ut} is the $k$-linear map
\begin{align*}
\mc F\oplus \mc G: (U\oplus U')\times\dots\times (U
\oplus U')&\to V\oplus V'\\
(u_1+u'_1,\dots,u_k+u'_k)
&\mapsto\mc F(u_1,\dots,u_k)+
\mc G(u'_1,\dots,u'_k).
\end{align*}
A $k$-linear map $\mc F: U\times\dots\times U\to V$ with $k
 \ge 2$ is
\emph{symmetric} (respectively, \emph{skew-symmetric}) if
\begin{equation*}\label{7yt}
\mc F(u_1,\dots,u_{i-1},
u_j,u_{i+1},\dots,
u_{j-1},
u_i,u_{j+1},\dots,u_k)
\end{equation*}
is equal to $\mc F(u_1,\dots,u_k)$ (respectively, $-\mc F(u_1,\dots,u_k)$)
for all $i,j\in\{1,\dots,k\}$ with $i<j$ and all $u_1,\dots,u_k\in U$.

The following theorem  generalizes \cite[Theorem 2]{bel-ser_mul}
about $k$-linear forms (i.e., with $\ff$ instead of $V$ and $V'$).

\begin{theorem}[proved in Section \ref{s4}]\label{ther}
Let
\[
\mc F: U\times\dots\times U\to V,\qquad\mc G: U'\times\dots\times U'\to V'
\]
be two $k$-linear maps with $k\ge 2$ over $\ff$.

\begin{itemize}
  \item[\rm(a)] Let $\ff$ be an algebraically closed field of characteristic different from $2$. Then
$\mc F$ and $\mc G$
over $\ff$ are symmetrically
equivalent if and only if they are  congruent. In particular,
if $\mc F$ and $\mc G$ are both symmetric or both skew-symmetric, then they are equivalent if and only if they are congruent.

  \item[\rm(b)]
Let $\ff$ be a real closed field. Then $\mc F$ and $\mc G$ are symmetrically equivalent if and only if there exists a
$k$-linear map $\mc H\oplus \mc K$  such that
$\mc F$ is congruent to
$\mc H\oplus \mc K$ and $\mc H\oplus -\mc K$
is congruent to $\mc G$.

\end{itemize}
\end{theorem}

\section{Proof of Theorem \ref{t1h}}
\label{s2}

We prove Theorem \ref{t1h} by the method that is developed in
\cite{ser}, described in details in \cite{ser1}, and is used in \cite{h-s2,hor-ser} and other articles. The reader is expected to be familiar with it. In order to better understand the nature of Theorem \ref{t1h},  in the next section we also give a direct proof of the statement (a).

The method developed in
\cite{ser} reduces the
problem of classifying systems of linear
maps and forms to the problem of
classifying systems of linear maps. Bilinear and sesquilinear forms, pairs of
symmetric, skew-symmetric, and Hermitian
forms, unitary and selfadjoint operators on
a vector space with indefinite scalar product
are classified in \cite{ser} over a field $\mathbb F$ of characteristic $\ne 2$ up to
classification of Hermitian forms over
finite extensions of $\mathbb F$ (and so they are fully classified over $\rr$ and
$\cc$).

Let $\ff$ be a field or skew field with a fixed involution $a\mapsto\bar a$. Each finite system consisting of vector spaces over $\ff$, and linear maps and sesquilinear (bilinear if $a\mapsto\bar a$ is the identity) forms between these spaces is considered in \cite{ser} as a representation of a graph with undirected and directed edges. Its vertices represent vector spaces, its undirected edges represent forms, and its directed edges represent linear maps. We give all forms and linear maps by their matrices if all the vector spaces are of the form $\ff^k$ with $k\in\{0,1,2,\dots\}$.

Let $\mc A=(A_1,\dots,A_t)$ and $\mc B=(B_1,\dots,B_t)$ be two
$t$-tuples
of $n\times n$ matrices over $\mathbb F$. Consider the corresponding representations \begin{equation}\label{ccn3}
\begin{split}
\xymatrix{{}\\{\mc A}:}\ \ \xymatrix{&{}\\
{\ff^n} \ar@{}[ur]|{\displaystyle \ddots}
\ar@{-}@(dl,l)^{A_1}
\ar@{-}@(l,ul)^{A_2}
\ar@{-}@(ul,u)^{A_3}
\ar@{-}@(r,rd)^{A_t}&{}
}
\qquad\qquad\qquad
\xymatrix{{}\\{\mc B}:}\ \ \xymatrix{&{}\\
{\ff^n} \ar@{}[ur]|{\displaystyle \ddots}
\ar@{-}@(dl,l)^{B_1}
\ar@{-}@(l,ul)^{B_2}
\ar@{-}@(ul,u)^{B_3}
\ar@{-}@(r,rd)^{B_t}&{}
}
\end{split}
\end{equation}
of the graph
\begin{equation*}\label{yyy3}
\begin{split}
\xymatrix{{}\\G:}\ \ \xymatrix{&{}\\
u \ar@{}[ur]|{\displaystyle \ddots}
\ar@{-}@(dl,l)^{\alpha_1}
\ar@{-}@(l,ul)^{\alpha_2}
\ar@{-}@(ul,u)^{\alpha_3}
\ar@{-}@(r,rd)^{\alpha_t}&{}
}\end{split}
\end{equation*}
with $t$ undirected loops. All matrices $A_i$ and $B_i$ in \eqref{ccn3} represent sesquilinear (bilinear if $a\mapsto\bar a$ is the identity) forms on the vector space $\ff^n$. The representations \eqref{ccn3}
are isomorphic if there exists a nonsingular matrix $S\in\ff^{n\times n}$ such that
\begin{equation*}\label{ewm3}
S^{\st}(A_1,\dots,A_t)S
=(B_1,\dots,B_t)
\end{equation*}
in which $S^{\st}:=S^T$ if $a\mapsto\bar a$ is the identity, and  $S^{\st}:=S^*$ if $a\mapsto\bar a$ is nonidentity.

The representations \eqref{ccn3} define the representations
\begin{equation}\label{6yr}
\un{\mc A}:\
\xymatrix@C=55pt{\ff^n
 \ar@/_/[r]_{A_1^{\st}}
 \ar@/_30pt/[r]_{A_t^{\st}}^{\cdots}
 \ar@/^/[r]^{A_1}
\ar@/^25pt/[r]^{A_t}_{\ldots}
&\ff^n}
\qquad\qquad\qquad
\un{\mc B}:\
\xymatrix@C=55pt{\ff^n
 \ar@/_/[r]_{B_1^{\st}}
 \ar@/_30pt/[r]_{B_t^{\st}}^{\cdots}
 \ar@/^/[r]^{B_1}
\ar@/^25pt/[r]^{B_t}_{\ldots}
&\ff^n}
\end{equation}
of the quiver with involution
\begin{equation*}\label{yyd}
\un G:\ \xymatrix@C=55pt{u
 \ar@/_/[r]_{\alpha_1^*}
 \ar@/_30pt/[r]_{\alpha_t^*}^{\cdots}
 \ar@/^/[r]^{\alpha_1}
\ar@/^25pt/[r]^{\alpha_t}_{\ldots}
&u^*}
\end{equation*}

Applying \cite[Theorem 1]{ser} (or \cite[Theorem 3.1]{ser1}) to representations of $G$ and $\un G$, we obtain the following:
\begin{itemize}
  \item[(A)] If\/ $\ff$ is an algebraically closed field of characteristic different from $2$, then $\mc A$ is isomorphic to $\mc B$ if and only if $\un{\mc A}$ is isomorphic to $\un{\mc B}$.

  \item[(B)] Let $\ff$ be an algebraically closed field with a fixed nonidentity involution, or a real closed field, or the skew field of quaternions over $\rr$ with a fixed nonidentity involution.
      \begin{itemize}
        \item
        %[(B1)]
        Each representation of $G$ is uniquely, up to isomorphisms of summands, decomposes into a direct sum of indecomposable representations.

        \item%[(B2)]
        If $\mc A$ is an indecomposable representation of $G$ and $\un{\mc A}$ is isomorphic to $\un{\mc B}$, then ${\mc A}$ is isomorphic to $\mc B$ or $-\mc B$.
      \end{itemize}
\end{itemize}

By the statement (B), if $\un{\mc A}$ is isomorphic to $\un{\mc B}$ and $\mc A$ is isomorphic to a direct sum $\mc A_1\oplus\dots\oplus \mc A_r$ of indecomposable representations, then the direct summands can be renumbered such that $\mc B$ is isomorphic to
$\mc A_1\oplus\dots\oplus\mc A_s \oplus-\mc A_{s+1}\oplus\dots\oplus-\mc A_r$,
in which $0\le s\le r$.

The representations \eqref{6yr} are isomorphic if there exist nonsingular matrices $P$ and $Q$ such that
\[
P(A_1,\dots,A_t,A_1^{\st}, \dots,A_t^{\st})Q=
(B_1,\dots,B_t,B_1^{\st}, \dots,B_t^{\st});
\]
that is, $P\mc A Q=\mc B$ and
$P\mc A^{\st} Q=\mc B^{\st}$.
By the last equality, $Q^{\st} \mc AP^{\st} =\mc B$. If $R:=Q^{\st}$, then
$P\mc A R^{\st}=R\mc AP^{\st} =\mc B$, which means that the matrix $t$-tuples $\mc A$ and $\mc B$ are symmetrically equivalent if the involution $a\mapsto\bar a$ is the identity, and they are symmetrically *equivalent if $a\mapsto\bar a$ is nonidentity.
Therefore, Theorem \ref{t1h} follows from (A) and (B).

\section{Direct proof of Theorem \ref{t1h}(a)}
\label{s3}

Let $\ff$ be an algebraically closed field of characteristic different from $2$. Suppose that two matrix $t$-tuples
$\mc A=(A_1,\dots,A_t)$ and $\mc B=(B_1,\dots,B_t)$ are symmetrically equivalent; that is, $P\mc AR^T=R\mc AP^T=\mc B$ for some nonsingular $P$ and $R$. Thus, $PA_iR^T=RA_iP^T=B_i$,
and so $P(A_i,A_i^T)R^T=(B_i,B_i^T)$ for each $i\in\{1,\dots,t\}$.

Let $Q:=(R^T)^{-1}$, then $P(A_i,A_i^T)=(B_i,B_i^T)Q$
and the diagram
\[
    \xymatrix@C=55pt@R=25pt{
{\circ}\ar@<-2pt>[d]_{A_i}
\ar@<2pt>[d]^{A^T_i}
\ar[r]^{Q}
&{\circ}\ar@<-2pt>[d]_{B_i}
\ar@<2pt>[d]^{B^T_i}
                      \\
{\circ}\ar[r]_{P}&{\circ}}
\]
is commutative.
Define the pair $(C_i,C_i^T)$ as follows:
\begin{equation}\label{asq}
\begin{split}
    \xymatrix@C=60pt@R=35pt{
{\circ}\ar@<-2pt>[d]_{A_i}
\ar@<2pt>[d]^{A^T_i}
\ar[r]_{Q}
\ar@/^1pc/@{-->}[rr]^{I_n}
&{\circ}\ar@<-2pt>[d]_{B_i}
\ar@<2pt>[d]^{B^T_i}
\ar[r]_{Q^{-1}}
&{\circ}\ar@<-2pt>[d]_{C_i}
\ar@<2pt>[d]^{C^T_i}
     &
\ar@{}[d]|{\qquad\qquad\displaystyle
\begin{aligned}
 C_i:=&Q^TB_iQ\\
 C_i^T\!=&Q^TB^T_iQ\\
 R:=&Q^TP
\end{aligned}
}
                      \\
{\circ}\ar[r]^{P}
\ar@/_1pc/@{-->}[rr]_{R}
&{\circ}
\ar[r]^{Q^T}&{\circ}
     &
}
\end{split}
\end{equation}
and obtain the commutative diagram
\begin{equation*}\label{jjj}
    \xymatrix@C=60pt@R=25pt{
{\circ}\ar@<-2pt>[d]_{A_i}
\ar@<2pt>[d]^{A^T_i}
\ar[r]^{I_n}
&{\circ}\ar@<-2pt>[d]_{C_i}
\ar@<2pt>[d]^{C^T_i}
\ar[r]^{R^T}
&{\circ}\ar@<-2pt>[d]_{A_i}
\ar@<2pt>[d]^{A^T_i}
                      \\
{\circ}\ar[r]_{R}
&{\circ}
\ar[r]_{I_n}&{\circ}}
\end{equation*}
Therefore,
\begin{equation*}\label{fff}
\begin{split}
    \xymatrix@C=60pt@R=25pt{
{\circ}\ar[d]_{A_i}
\ar[r]^{R^T}
&{\circ}\ar[d]^{A_i}
                      \\
{\circ}\ar[r]_{R}
&{\circ}}
              \qquad
\xymatrix@C=60pt@R=25pt
{{}\ar@{}[d]|
{\displaystyle\text{and so}}\\{}}
              \qquad
    \xymatrix@C=60pt@R=25pt{
{\circ}\ar[d]_{A_i}
\ar[r]^{f(R^T)}
&{\circ}\ar[d]^{A_i}
                      \\
{\circ}\ar[r]_{f(R)}
&{\circ}}
\end{split}
\end{equation*}
for each polynomial $f\in\mathbb F[x]$.

Since $A_iR^T=C_i$,
the diagram
\begin{equation*}\label{ddb}
\begin{split}
    \xymatrix@C=60pt@R=30pt{
{\circ}\ar[d]_{A_i}
\ar[r]_{f(R^T)}
\ar@/^1pc/@{-->}[rr]^{R^{-T}f(R^T)}
&{\circ}\ar[d]_{A_i}
\ar[r]_{R^{-T}}
&{\circ}\ar[d]^{C_i}
                      \\
{\circ}\ar[r]^{f(R)}
\ar@/_1pc/@{-->}[rr]_{f(R)}
&{\circ}
\ar[r]^{I_n}&{\circ}}
\end{split}
\end{equation*}
is commutative. Hence $f(R)A_i=C_i\bigl(R^{-T}f(R^T) \bigr)$ and
\[C_i=f(R)A_i
\bigl(R^{-T}f(R^T)\bigr)^{-1}
=f(R)A_i\bigl(Rf(R)^{-1}\bigr)^{T}.\]

By \cite[Section VIII, \S\,6]{gan} or \cite[Section 6.4]{horn}, there exists $f\in\mathbb F[x]$ such that
$f(R)^{2}=R$. Since $R$ is nonsingular, $f(R)$ is nonsingular too, $f(R)=Rf(R)^{-1}$, and $C_i=f(R)A_if(R)^{T}$. By \eqref{asq}, $B_i=Q^{-T}C_iQ^{-1}
=Q^{-T}f(R)A_if(R)^{T}Q^{-1}$. Taking $S:=Q^{-T}f(R)$, we obtain $B_i=SA_iS^T$. The matrix $S$ is the same for all $A_1,\dots,A_t$. Therefore, $\mc B=S\mc AS^T$, i.e. $\mc A$ is congruent to $\mc B$.

\section{Proof of Theorem \ref{ther}}
\label{s4}

\begin{definition}
Let $\mathscr  F:=(\mc F_1,\dots,\mc F_t)$ and $\mathscr  G:=(\mc G_1,\dots,\mc G_t)$ be two $t$-tuples of $k$-linear forms
\begin{equation}\label{gvt2}
\mc F_1,\dots,\mc F_t: U\times\dots\times U\to \ff,\qquad\mc G_1,\dots,\mc G_t: U'\times\dots\times U'\to \ff
\end{equation}
on vector spaces $U$ and $U'$ over a field $\ff$. We say that
$\mathscr  F $ and $\mathscr  G$ are
\begin{itemize}
  \item
\emph{symmetrically equivalent}
if there exist linear bijections $\varphi_1,\dots,\varphi_k:U\to U'$ such that
\begin{equation}\label{jndn}
\forall \sigma \in S_k,\, u_1,\dots,u_k\in U:\ \mc G_{\ell}(\varphi_{\sigma(1)}u_1,
\dots,\varphi_{\sigma(k)}u_k )= \mc F_{\ell}(u_1,\dots,u_k)
\end{equation}
for each $\ell=1,\dots,t$;

  \item
\emph{congruent} if
there exists a linear bijection $\varphi:U\to U'$ such that
\begin{equation}\label{jndj}
\forall u_1,\dots,u_k\in U:\ \mc G_{\ell}(\varphi u_1,
\dots,\varphi u_k )=\mc F_{\ell}(u_1,\dots,u_k)
\end{equation}
for each $\ell=1,\dots,t$.
\end{itemize}
\end{definition}

Let $\mc G: V\times\dots \times V\to \ff$ be a $k$-linear form. A linear map $\tau:
V\to V$ is \emph{$\mc G$-selfadjoint} if
\begin{equation*}\label{5}
\mc G(v_1\dots,v_{i-1}, \tau
v_i,v_{i+1}\dots,v_n)=
\mc G(v_1\dots,v_{j-1},\tau
v_j,v_{j+1}\dots,v_n)
\end{equation*}
for all $v_1,\dots,v_n\in V$ and all
$i,j$.

The proof of Theorem \ref{ther} is based on the following theorem, which is given in  \cite[Theorems 2(a) and 3]{bel-ser_mul} for $t=1$.

\begin{theorem}\label{prop}
Let $\mathscr  F=(\mc F_1,\dots,\mc F_t)$ and $\mathscr G=(\mc G_1,\dots,\mc G_t)$ be two $t$-tuples of $k$-linear forms \eqref{gvt2} over $\ff$.

\begin{itemize}
  \item[\rm(a)]
Let $\ff$ be an algebraically closed field of characteristic different from $2$. Then $\mathscr  F$ and $\mathscr  G$ are symmetrically equivalent  if and only if they are congruent.

\item[\rm(b)]
Let $\ff$ be a real closed field. Then $\mathscr  F$ and $\mathscr  G$ are symmetrically equivalent if and only if there exists a $t$-tuple $\mathscr  H\oplus \mathscr  K$  of $k$-linear forms such that
$\mathscr  F$ is congruent to
$\mathscr  H\oplus \mathscr  K$ and $\mathscr  H\oplus -\mathscr  K$
is congruent to $\mathscr  G$.
\end{itemize}
\end{theorem}

\begin{proof}
(a) Let us show that this statement can be proved as \cite[Theorem 2(a)]{bel-ser_mul}. Let $\mathscr  F$ and $\mathscr  G$ be symmetrically equivalent; that is, there exist linear bijections $\varphi_1,\dots,\varphi_k:U\to U'$ satisfying \eqref{jndn}. Let $\varphi_1=\dots=\varphi_r\ne \varphi_{r+1}$ for some $r<k$.
It suffices to prove that $\varphi_1,\dots,\varphi_k$ in \eqref{jndn} can be replaced by $\psi_1,\dots,\psi_k$ such that
\begin{equation}\label{bnu}
\psi_1=\dots=\psi_{r+1}.
\end{equation}
We have
\begin{equation}\label{4}
  \mc G_{\ell}(\dots,\varphi_r u_i,\dots,
  \varphi_{r+1}u_j,\dots)=
  \mc G_{\ell}(\dots,\varphi_{r+1}u _i,\dots,\varphi_r u_j,\dots)
\end{equation}
for all $\ell=1,\dots,t$, $1\le i<j\le k$, $u_i,u_j\in U$, and all elements
$v_1,\dots,v_{i-1},
v_{i+1},\dots,v_{j-1},
v_{j+1},\dots,v_k\in V$ that are represented by the points.
Write $v_i:=\varphi_{r+1}u_i$,
$v_j:=\varphi_{r+1}u_j$, and consider the linear map $
\tau:=\varphi_r\varphi_{r+1}^{-1}:
V\to V $. By
\eqref{4},
\begin{equation*}\label{ucd}
  \mc G_{\ell}(\dots,\tau v_i,\dots,
  v_j,\dots)=
  \mc G_{\ell}(\dots,v _i,\dots,\tau v_j,\dots).
\end{equation*}
Therefore, $\tau $ is {$\mc G_{\ell}$-selfadjoint} for each $\ell$.

Let $\lambda_1,\dots,\lambda_s$ be
all the distinct eigenvalues of $\tau$,
and let \[V=V_1\oplus\dots\oplus V_s,\qquad \tau_i:=\tau|V_i
\text{ has the single eigenvalue $\lambda_i$}
\] be the
decomposition of $V$ into the direct
sum of $\tau$-invariant subspaces. By \cite[Lemma 6]{bel-ser_mul},
\[
\mc G_{\ell}=\mc G_{\ell 1}\oplus\dots\oplus \mc G_{\ell s},\qquad
\mc G_{\ell i}:=\mc G_{\ell}|V_i
\]
for each $\ell=1,\dots, t$.
By \cite[Lemma
4(a)]{bel-ser_mul}, there exists $f_i(x)\in\ff[x]$ such
that
$f_i(\tau_i)^{r+1}=\tau_i^{-1}$.
Then
\[
\rho:=f_1(\tau_1)\oplus\dots\oplus
f_s(\tau_s):  V\to V
\]
is $\mc G_{\ell}$-selfadjoint and
$\rho^{r+1}=\tau^{-1}$.

Define the linear bijections
\begin{equation*}\label{6ab}
  \psi_1=\dots=\psi_{r+1}:=
  \rho\varphi_r,
\qquad \psi_{r+2}:=\varphi_{r+2},\dots,
\psi_k:=\varphi_k.
\end{equation*}
Since $\rho$ is $\mc G_{\ell}$-selfadjoint and
\[
 \rho^{r+1}\varphi_r
=\tau^{-1}\varphi_r =(\varphi_r
\varphi_{r+1}^{-1})^{-1} \varphi_r=
\varphi_{r+1},
\]
we have
\begin{align*}
\mc G_{\ell}(\psi_1u_1,\dots, \psi_nu_n)
 &=
\mc G_{\ell}(\rho\varphi_r u_1,\dots, \rho\varphi_r
u_r, \rho\varphi_r u_{r+1},
\varphi_{r+2}
u_{r+2},\dots,\varphi_k u_k)
 \\&=
\mc G_{\ell}(\varphi_r u_1,\dots, \varphi_r u_r,
\rho^{r+1}\varphi_r u_{r+1},
\varphi_{r+2}
u_{r+2},\dots,\varphi_k u_k)
 \\&=
\mc G_{\ell}(\varphi_1u_1,\dots, \varphi_k
u_k)= \mc F_{\ell}(u_1,\dots,u_k).
\end{align*}
The equality
$
\mc G_{\ell}(\psi_{\sigma(1)}u_1,
\dots,\psi_{\sigma(k)}u_k )= \mc F_{\ell}(u_1,\dots,u_k)
$ for an arbitrary $\sigma \in S_k$ is
proved analogously. Therefore, $\psi_1,\dots,\psi_k$ can be used instead of $\varphi_1,\dots,\varphi_k$ and \eqref{bnu} holds.

(b) This statement can be proved as \cite[Theorem 3]{bel-ser_mul}.
\end{proof}

\begin{proof}[Proof of Theorem \ref{ther}]
Let us prove (a); the statement (b) is proved analogously.

Let $k$-linear maps \[\mc F: U\times\dots\times U\to V,\qquad \mc G: U'\times\dots\times U'\to V'\] with $k\ge 2$ be
symmetrically equivalent; that is, \eqref{jnd} holds.
Let  $f_1,\dots,f_t$ be a basis of $V$. Then $f'_1:=\psi f_1$, \dots, $f'_t:=\psi f_t$ is a basis of $V'$. Let \[\mathscr  F:=(\mc F_1,\dots,\mc F_t),\qquad \mathscr G:=(\mc G_1,\dots,\mc G_t)\] be two $t$-tuples of $k$-linear forms  \eqref{gvt2} defined by
\begin{equation}
\label{bgq}
\begin{split}
\mc F(u_1,\dots,u_k)&=
\mc F_1(u_1,\dots,u_k)f_1 +\dots+\mc F_t(u_1,\dots,u_k)f_t
\\
\mc G(u'_1,\dots,u'_k)&=
\mc G_1(u'_1,\dots,u'_k)f'_1 +\dots+\mc G_t(u'_1,\dots,u'_k)f'_t
\end{split}
\end{equation}
for all $u_1,\dots,u_k\in U$ and $u'_1,\dots,u'_k\in U'$.
Substituting \eqref{bgq} in \eqref{jnd}, we obtain \eqref{jndn} for each $\ell=1,\dots,t$. Therefore, $\mathscr F$ and $\mathscr G$ are
symmetrically equivalent.
By Theorem \ref{prop},
they are congruent; that is, \eqref{jndj} holds for some linear bijection $\varphi:U\to U'$ and all $\ell=1,\dots,t$. By \eqref{bgq}, \eqref{jnd8} holds for this $\varphi$ and for $\psi$ from \eqref{jnd}. Hence, $\mc F$ and $\mc G$ are congruent.

In particular, if $\mc F$ and $\mc G$ are both symmetric or both skew-symmetric and they are equivalent, then they are congruent since $\mc F: U\times\dots\times U\to V$ is symmetric if
\[
\forall \sigma \in S_k,\, u_1,\dots,u_k\in U:\ \mc F(u_{\sigma(1)},
\dots,u_{\sigma(k)})= \mc F(u_1,\dots,u_k);
\]
it is skew-symmetric  if
\[
\forall \sigma \in S_k,\, u_1,\dots,u_k\in U:\ \mc F(u_{\sigma(1)},
\dots,u_{\sigma(k)})=\sign(\sigma) \mc F(u_1,\dots,u_k).
\]
For example, if $\mc F$ and $\mc G$ are skew-symmetric and \eqref{jnd7} holds, then
\eqref{jnd} holds too since
\begin{align*}
\mc G(\varphi_{\sigma(1)}u_1,
\dots,\varphi_{\sigma(k)}u_k )
     &=
\sign(\sigma^{-1})\mc
G(\varphi_1
u_{\sigma^{-1}(1)},\dots,
\varphi_k
u_{\sigma^{-1}(k)})
\\
&=
\sign(\sigma^{-1})\psi\mc
F(
u_{\sigma^{-1}(1)},\dots,
u_{\sigma^{-1}(k)})
\\
&=\psi \mc F(u_1,\dots,u_k)
\end{align*}
for each $\sigma \in S_k$ and all $u_1,\dots,u_k\in U$.
Thus, $\mc F$ and $\mc G$ are symmetrically equivalent, and so they are congruent.
\end{proof}

\begin{remark} \label{bvc}
Theorem \ref{ther} in the case of bilinear maps (i.e., for $k=2$) follows from Theorem \ref{t1h}.
Let us prove this fact for the case of an algebraically closed field $\ff$ of characteristic
different from 2.
Let bilinear maps \[\mc F: U\times U\to V,\qquad \mc G: U'\times U'\to V'\] over  $\ff$ be   symmetrically equivalent; that is,
\begin{equation}\label{nnh}
\forall x,y\in U:\ \mc G(\varphi_1x,\varphi_2y) =\mc G(\varphi_2x,\varphi_1y)=\psi
\mc F(x,y)
\end{equation}
for some linear bijections $\varphi_1,\varphi_2:U\to U'$ and $\psi:V\to V'$.

Choose bases $e_1,\dots,e_n$ in $U$ and $f_1,\dots,f_t$ in $V$.
Write
\[
\mc F(e_i,e_j)
=a_{ij}^{(1)}f_1
+\dots+a_{ij}^{(t)}f_t,\qquad
a_{ij}^{(1)},\dots, a_{ij}^{(t)}\in\ff,\] and
define the $n\times n$ matrices \[A_1:=[a_{ij}^{(1)}],\ \dots,\ A_t:=[a_{ij}^{(t)}].\] Then \[
\mc F(x,y)
=([x]_e^TA_1[y]_e)f_1
+\dots+([x]_e^TA_t[y]_e)f_t.
\]
Analogously,
\[
 \mc G(x,y)
=([x]_{e'}^TB_1[y]_{e'})f'_1
+\dots+([x]_{e'}^TB_t[y]_{e'})f'_t
\]
in some bases of $U'$ and $V'$.

Let us prove that the matrix $t$-tuples
 \begin{equation}\label{xxx}
 \mc A:=(A_1,
\dots,A_t),\qquad
\mc B:=(B_1,\dots,B_t)
\end{equation}
are weakly symmetrically equivalent.
Let $\Phi_1$, $\Phi_2$, and $\Psi=[\lambda_{ij}]_{i,j=1}^t$ be the matrices of $\varphi_1$, $\varphi_2$ and $\psi$ in these bases. By \eqref{nnh},
\[
\big([x]_{e}^T\Phi_1^TB_1\Phi_2[y]_{e},
\dots,[x]_{e}^T\Phi_1^TB_t
\Phi_2[y]_{e}\big)^T
=\Psi \big([x]_e^TA_1[y]_e,\dots,
[x]_e^TA_t[y]_e\big)^T.
\]
Hence,
$
[x]_{e}^T\Phi_1^T(B_1,
\dots,B_t)\Phi_2[y]_{e}
=[x]_{e}^T\widetilde{\Psi}[y]_{e}
$
where
\[
\widetilde{\Psi}:=
(\lambda_{11}A_1
+\dots+\lambda_{1t}A_t,\ \dots,\ \lambda_{t1}A_1
+\dots+\lambda_{tt}A_t).
\]
Thus, $\Phi_1^T(B_1,
\dots,B_t)\Phi_2
=\widetilde{\Psi}$.
By \eqref{nnh}, we can take $\varphi_2$ and $\varphi_1$ instead of $\varphi_1$ and $\varphi_2$, and obtain $\Phi_2^T(B_1,
\dots,B_t)\Phi_1
=\widetilde{\Psi}
$. Therefore, the $t$-tuples  \eqref{xxx} are weakly symmetrically equivalent. By Corollary \ref{t1n1}(a), they are weakly congruent; that is, we can take $\Phi_1=\Phi_2$. Then $\varphi_1=\varphi_2$, and so $\mc F$ and $\mc G$ are congruent.
\end{remark}

\section*{Acknowledgements}
V. Futorny was supported by the CNPq (304467/2017-0) and the FAPESP (2018/23690-6).  V.V.~Sergeichuk was
supported by FAPESP (2018/24089-4). The work was started when V.V.~Sergeichuk visited the Ben-Gurion University of the Negev.

\end{document}